\documentclass[letterpaper, 10 pt, conference]{ieeeconf}  %

\IEEEoverridecommandlockouts                              %

\usepackage{graphicx}
\usepackage{epsfig} %
\usepackage{mathptmx} %
\usepackage{times} %
\usepackage{amsmath} %
\usepackage{amssymb}  %
\usepackage[]{algorithm2e}
\usepackage{cite}
\usepackage{hyperref}
\usepackage{booktabs}
\usepackage{array, multirow}

\DeclareMathOperator{\Dom}{Dom}

\newtheorem{thm}{Theorem}
\newtheorem{lemma}{Lemma}

\newtheorem{remark}{Remark}

\title{\LARGE \bf
	A hybrid proximal generalized conditional gradient method and application to total variation parameter learning
}

\author{Kristian Bredies$^{1}$ Enis Chenchene$^{1}$ Alireza Hosseini$^{2}$
	\thanks{$^{1}$Institute of Mathematics and Scientific Computing, University of Graz, Graz, Austria.
			{\tt\small kristian.bredies@uni-graz.at, enis.chenchene@uni-graz.at}}%
	\thanks{$^{2}$School of Mathematics, Statistics and Computer Science, College of Science, University of Tehran, Tehran, Iran.
			{\tt\small hosseini.alireza@ut.ac.ir}}%
}

\begin{document}

\maketitle
\thispagestyle{empty}
\pagestyle{empty}

\begin{abstract}

	In this paper we present a new method for solving optimization problems involving the sum of two proper, convex, lower semicontinuous functions, one of which has Lipschitz continuous gradient. The proposed method has a hybrid nature that combines the usual forward--backward and the generalized conditional gradient method. We establish a convergence rate of $o(k^{-1/3})$ under mild assumptions with a specific step-size rule and show an application to a total variation parameter learning problem, which demonstrates its benefits in the context of nonsmooth convex optimization.

\end{abstract}

\section{INTRODUCTION}

\noindent Given a Hilbert space $H$, the generalized conditional gradient method is a powerful tool to solve
\begin{equation}\label{eq:minimization_problem}
	\min_{u\in H} \ f(u) + g(u),
\end{equation}
where $f$ and $g$ are suitable convex, proper, lower semicontinuous functions \cite{Bredies2009, Bredies2022, linearcondgrad, Bredies2013InversePI}. The general iteration reads, for $k \in \mathbb{N}$,
\begin{equation}\label{eq:fw_y}
	v^{k} \in \arg \min_{v\in H} \ \langle \nabla f(u^k), v \rangle +g(v),
\end{equation}
and the new update is then obtained as
\begin{equation*}
	u^{k+1} = u^k + \theta_k (v^{k}-u^{k}),
\end{equation*}
where $\theta_k$ is a \emph{step-size} that can be obtained, e.g., by line-search, backtracking or satisfying a certain step-size rule. The generalized conditional gradient method has an interesting connection with the more popular forward--backward method \cite{lions}. Indeed, adding and removing in \eqref{eq:minimization_problem} a quadratic term
\begin{equation*}
	\min_{u\in H} \ f(u) - \tfrac{1}{2} \| u\|_P^2 +g(u)+\tfrac{1}{2} \| u\|_P^2,
\end{equation*}
where $\|u\|^2_P := \langle u, Pu \rangle$ and $P:H\to H$ is a self-adjoint, positive definite, bounded, linear operator, then applying the generalized conditional gradient with $F(u) :=f(u) - \frac{1}{2} \|u\|_P^2$, $G(u) := g(u)+\frac{1}{2} \|u\|_P^2$ and $\theta_k = 1$ for all $k$ gives
\begin{align*}
	u^{k+1} = v^{k} & \in \arg \min_{v\in H} \ \langle \nabla f(u^{k})-Pu^k, v\rangle + g(v)+\tfrac{1}{2}\|v\|_P^2 \\
	                & = \arg \min_{v\in H} \ g(v) + \tfrac{1}{2} \|v-(u^k-P^{-1}\nabla f(u^k))\|_P^2.
\end{align*}
This coincides with the celebrated (preconditioned) forward--backward method with respect to the metric induced by $P$. Note that, of course, if $P=0$, one retrieves again the generalized conditional gradient method.
In this work, we investigate the in-between scenario assuming that $P$ is only positive semidefinite, which will lead to a hybrid method that we call Hybrid Proximal Generalized Conditional Gradient (HPGCG). Note as well that, in the same spirit of \cite{degeneratePPP}, the proposed method can also be understood as an instance of a \emph{degenerate} forward--backward method. In Section \ref{sec:convergence}, we present the algorithm as well as its convergence analysis. In Section \ref{sec:learning} we show that the proposed method is particularly suitable to solve a total variation (TV) parameter learning problem in  mathematical image reconstruction.

\section{THE PROPOSED METHOD AND ITS CONVERGENCE ANALYSIS}\label{sec:convergence}
Throughout, we assume that $f:H\rightarrow \mathbb{R}$ is convex and Fréchet differentiable with Lipschitz continuous gradient, $g:H\rightarrow \mathbb{R}\cup \{\infty\}$ is proper, convex and lower semicontinuous, $f+g$ is coercive. Under these assumptions, problem \eqref{eq:minimization_problem} always admits an optimal solution. Eventually, we assume that $P:H\rightarrow H$ is a bounded, linear, positive semidefinite operator such that $\frac{1}{2}\|\cdot\|_P^2+g$ is strongly coercive.

\subsection{HPGCG algorithm}
In this section, we present the proposed HPGCG method along with our step-size choice. To do so, we first need to introduce some notation. First, we define $v(u)\in H$ as any minimizer of
\begin{equation}\label{eq:def_v}
	\min_{v\in H} \ \langle\nabla f(u), v\rangle - \langle u,v\rangle_P+\tfrac{1}{2}\|v\|_P^2+g(v),
\end{equation}
which always admits an optimal solution due to the strong coercivity of $\frac{1}{2}\|\cdot\|_P^2+g$. Then, we set
\begin{equation*}
	H_u(v):=\langle\nabla f(u),v\rangle-\langle u,v\rangle_P+\tfrac{1}{2}\|v\|_P^2+g(v),
\end{equation*}
and define $D(u) := H_u(u)-\inf \ H_u$. Note that
\begin{equation*}
	D(u)=\langle\nabla f(u), u-v(u)\rangle+g(u)-g(v(u))-\tfrac{1}{2}\|u-v(u)\|_P^2.
\end{equation*}
We further fix a function $D_f:H\times H\rightarrow \mathbb{R}$ with the following properties:
\begin{itemize}
	\item $D_f(u,v)\geq f(v)-f(u)-\langle\nabla f(u),v-u\rangle$,
	\item $D_f(u,u+\theta(v-u))\leq \theta^2 D_f(u,v)$, for $\theta\in [0,1],$
	\item $D_f$ is bounded on bounded sets.
\end{itemize}
A simple choice for $D_f$ is given by $D_f(u,v)=\frac{L}{2}\|u-v\|_2^2$, where $L$ is a Lipschitz constant of $\nabla f$. However, if $f(u)=\frac{1}{2}\|R u-q\|^2,$ where $R: H\rightarrow K$ is a linear and bounded operator, $K$ a Hilbert space and $q\in K$, we can also pick $D_f(u,v)=\frac{1}{2}\|R(u-v)\|^2$, which does not require the knowledge of $L$.

Inspired from \cite{Lorenz2007}, we can now present our step-size rule, which is shown in Algorithm \ref{algorithm1}.
\begin{algorithm}[t]
	\SetAlgoLined
	\textbf{Initialize}: $u^0\in \Dom(G)$\\
	\For{$k=0,1,\dots$}{
		$v^k = v(u^k)$ according to \eqref{eq:def_v}\\
		Update the step-size as
		\begin{equation}\label{eq:step-size}
			\theta_k = \min \left(1,\frac{D(u^k)+\frac{1}{2}\|u^k-v^k\|_P^2}{2D_f(u^k,v^k)}\right)
		\end{equation}
		$u^{k+1} = u^k+\theta_k(v^k-u^k)$
	}
	\caption{HPGCG algorithm for solving \eqref{eq:minimization_problem}.}\label{algorithm1}
\end{algorithm}
\subsection{Convergence Analysis}
First, for every $u\in H$, let
\begin{equation*}
	r(u):=(f+g)(u)-\inf \ (f+g).
\end{equation*}
We have the following result.
\begin{lemma}\label{lem1}
	Let $u\in H$, $v = v(u)$ according to \eqref{eq:def_v}, and $\theta\in[0,1].$ Then, for every optimal solution $u^*$ of    \eqref{eq:minimization_problem} we have
	\begin{enumerate}
		\item $\displaystyle D(u)+\tfrac{1}{2}\|u-v\|_P^2\geq r(u)-\|u-v\|_P\|u^*-v\|_P$,
		\item the following holds
		      \begin{align*}
			      \hspace{-1cm}\theta & \left(D(u)+\tfrac{1}{2}\|u-v\|_P^2\right)\geq \theta r(u) -\tfrac{1}{2}\|u-u^*\|_P^2        \\
			                          & +\tfrac{1}{2}\|(u+\theta(v-u))-u^*\|_P^2+\tfrac{\theta}{2}\left(2-\theta\right)\|u-v\|_P^2,
		      \end{align*}
		\item $D(u)\geq 0$ for every $u \in H$ and $D(u)=0$ if and only if $u$ is an optimal solution to \eqref{eq:minimization_problem}.
	\end{enumerate}
\end{lemma}
\begin{proof}
	From the optimality conditions for \eqref{eq:def_v} and the subgradient inequality, we have
	\begin{equation}\label{eq:lem_1proof1}
		g(v)+\langle u-v,u^*-v\rangle_P-\langle\nabla f(u), u^*-v\rangle\leq g(u^*).
	\end{equation}
	Consequently, \eqref{eq:lem_1proof1} and the convexity of $f$ yield
	\begin{align}\label{eq:lem_1proof2}
		D(u) & \geq\langle u-v, u^*-v\rangle_P+\langle\nabla f(u),u-u^*\rangle \nonumber \\
		     & \quad+g(u)-g(u^*)-\tfrac{1}{2}\|u-v\|_P^2\nonumber                        \\
		     & \geq r(u)+\langle u-v, u^*-v\rangle_P-\tfrac{1}{2}\|u-v\|_P^2,
	\end{align}
	which implies 1) via Cauchy--Schwarz. From \eqref{eq:lem_1proof2} and the polarization identity, we get
	\begin{align*}
		\theta\big(D(u) & +\tfrac{1}{2}\|u-v\|_P^2\big)\geq \theta r(u) +\langle \theta(u-v),u^*-v\rangle_P    \\
		                & =\theta r(u) +\langle \theta(u-v), u^*-u\rangle_P+\theta\|u-v\|_P^2                  \\
		                & =\theta r(u) +\tfrac{1}{2}\|(u+\theta(v-u))-u^*\|_P^2-\tfrac{\theta^2}{2}\|u-v\|_P^2 \\
		                & \quad-\tfrac{1}{2}\|u-u^*\|_P^2+\theta\|u-v\|_P^2,
	\end{align*}
	which proves part 2). By definition, $D\geq 0.$ As $v$ is a minimizer of $H_u,$ $D(u)=0$ if and only if
	\begin{equation}\label{lemma1-1}
		u\in \arg \min_{v\in H} \langle\nabla f(u),v\rangle-\langle u,v \rangle_P+\tfrac{1}{2}\|v\|_P^2+g(v),
	\end{equation}
	which, from optimality conditions, is equivalent to $u$ being a minimizer of $f+g.$ This completes the proof of part 3).
\end{proof}
\begin{lemma}\label{lem:lem2}
	Let $\{u^k\}_k$ be a sequence generated by Algorithm \ref{algorithm1}, then for $k\geq 0$, we get
	\begin{equation}\label{eq:lem2}
		r(u^{k+1})-r(u^k)\leq -\tfrac{\theta_k}{2}\left(D(u^k)+\tfrac{1}{2}\|u^k-v^k\|_P^2\right).
	\end{equation}
	\begin{proof}
		By the definition of $r$, the first property of $D_f$ and convexity of  $g$, we get
		\begin{align*}
			r & (u^{k+1})-r(u^k)                                                                      \\
			  & = f(u^k+\theta_k(v^k-u^k))-f(u^k)-\theta_k\langle \nabla f(u^k), v^k-u^k\rangle       \\
			  & \quad + g(u^k+\theta_k(v^k-u^k))-g(u^k)+\theta_k\langle \nabla f(u^k), v^k-u^k\rangle \\
			  & \leq D_f(u^k,u^k+\theta_k(v^k-u^k))-\tfrac{\theta_k}{2}\|u^k-v^k\|_P^2                \\
			  & \quad +\theta_k\left(g(v^k)-g(u^k)-\langle\nabla f(u^k), u^k-v^k\rangle
			+\tfrac{1}{2}\|u^k-v^k\|_P^2\right).
		\end{align*}
		Consequently, using the definition of $D(u^k)$ and the second property of $D_f$ we obtain
		\begin{align}\label{eq:lem2-2}
			r & (u^{k+1})-r(u^k)                                                                               \\
			  & \leq-\theta_k\left(D(u^k)+\tfrac{1}{2}\|u^k-v^k\|_P^2\right)+\theta_k^2 D_f(u^k,v^k).\nonumber
		\end{align}
		Now, if $D(u^k)+\frac{1}{2}\|u^k-v^k\|_P^2\geq 2D_f(u^k,v^k)$, then $\theta_k=1$ and \eqref{eq:lem2-2} turns easily into \eqref{eq:lem2}. Otherwise, $\theta_k= \frac{D(u^k)+\frac{1}{2}\|u^k-v^k\|_P^2}{2D_f(u^k,v^k)}$, which again leads from \eqref{eq:lem2-2} to \eqref{eq:lem2}.
	\end{proof}
\end{lemma}
\begin{lemma}\label{lem:lem3}
	The sequences $\{u^k\}_k$ and $\{v^{k}\}_k$ produced by Algorithm \ref{algorithm1} are bounded.
\end{lemma}
\begin{proof}
	From Lemma \ref{lem:lem2}, and 3) of Lemma \ref{lem1}, $\{r(u^k)\}_k$ is non-increasing, thus bounded. As $f+g$ is coercive and $\nabla f$ is Lipschitz continuous, $\{u^k\}_k$ and $\{\nabla f(u^k)\}_k$ are bounded as well. As $v^k$ minimizes $H_{u^k}$, optimality conditions yield
	\begin{equation*}
		v^k \in (P+\partial g)^{-1}(Pu^k-\nabla f(u^k)).
	\end{equation*}
	Since $P+\partial g$ is strongly coercive, $(P+\partial g)^{-1}$ maps bounded sets to bounded sets, see, e.g., \cite[Theorem 3.3]{bbc}. Thus, $\{v^k\}_k$ is bounded.
\end{proof}
\begin{lemma}\label{lem:lem4} Let $\{u^k\}_k$ and $\{v^k\}_k$ be generated by Algorithm \ref{algorithm1}, then we have\vspace{-0.2cm}
	\begin{equation}\label{eq:lem4}
		\lim_{k\rightarrow\infty} \ D(u^k)+\tfrac{1}{2}\|u^k-v^k\|_P^2=0.
	\end{equation}
\end{lemma}
\vspace{0.2cm}
\begin{proof}
	For the sake of notation, we set $\xi_k:=D(u^k)+\frac{1}{2}\|u^k-v^k\|_P^2$ for all $k \in \mathbb{N}$. Let $k \in \mathbb{N}$, and note that if $\theta_k=1$, then, from \eqref{eq:lem2}, we get $r(u^{k+1})-r(u^k)\leq -\frac{1}{2}\xi_k$. If $\theta_k<1$, since $\{u^k\}_k$ and $\{v^k\}_k$ are bounded, $\{D_f(u^k,v^k)\}_k$ is bounded as well (Lemma \ref{lem:lem3} and third property of $D_f$), and thus there exists a $C>0$, such that $\theta_k \geq 2C\xi_k$. Consequently, using again \eqref{eq:lem2}, we get $r(u^{k+1})-r(u^k)\leq -C\xi_k^2$. In both cases, we obtain
	\begin{equation}\label{eq:lem4_2}
		r(u^{k+1})-r(u^k)\leq -\min \big\{  \tfrac{1}{2}\xi_k, \ C\xi_k^2 \big\}.
	\end{equation}
	Thus, the right-hand-side of \eqref{eq:lem4_2} is summable and, in particular, \eqref{eq:lem4} holds.
\end{proof}
\begin{thm}\label{thm:1}
	Let $\{u^k\}_k$ be generated by Algorithm \ref{algorithm1}. Then, $\{r(u^k)\}_k$ converges monotonically to zero with rate $o(k^{-1/3})$.
\end{thm}
\begin{proof}
	From Lemma \ref{lem:lem2}, part 2) of Lemma \ref{lem1} and the definition of $u^{k+1}$, we have
	\begin{align*}
		r & (u^{k})-r(u^{k+1})                                                          \\
		  & \geq \tfrac{\theta_k}{2}\left(D(u^k)+\tfrac{1}{2}\|u^k-v^k\|_P^2\right)     \\
		  & \geq \tfrac{\theta_k}{2}r(u^k)+\tfrac{1}{4}\|u^{k+1}-u^*\|_P^2-\tfrac{1}{4}
		\|u^k-u^*\|_P^2.
	\end{align*}
	In particular, $\{r(u^k)\}_k$ is monotonically non-increasing. Therefore, for each $n\in \mathbb{N}$,
	\begin{equation*}	\sum_{k=0}^{n-1}\theta_k r(u^k)\leq 2r(u^0)+\tfrac{1}{2}\|u^0-u^*\|_P^2\leq C,
	\end{equation*}
	for some $C>0$. Hence, $\{\theta_kr(u^k)\}_k$ is summable. Now, if $\theta_k < 1$ then, using part 1) of Lemma \ref{lem1} and boundedness of $\{D_f(u^k, v^k)\}_k$ (cf., proof of Lemma \ref{lem:lem4}), we get
	\begin{equation}\label{eq:proof_thm1_1}
		\begin{aligned}
			\theta_k & \geq C\left(D(u^k)+\tfrac{1}{2}\|u^k-v^k\|_P^2\right) \\
			         & \geq C(r(u^k)-\|u^k-v^k\|_P\|u^*-v^k\|_P),
		\end{aligned}
	\end{equation}
	for some $C > 0$. From the first inequality in \eqref{eq:proof_thm1_1}, since $D(u^k)\geq 0$ by part 3) of Lemma \ref{lem:lem2}, we get
	\begin{equation}\label{eq:thm1_proof2}
		\|u^k-v^k\|_P\leq \sqrt{2C^{-1}} \sqrt{\theta_k},
	\end{equation}
	and from the second, we get
	\begin{equation}\label{eq:thm1_proof1}
		r(u^k) \leq C^{-1}\theta_k+\|u^k-v^k\|_P\|u^*-v^k\|_P.
	\end{equation}
	Thus, using \eqref{eq:thm1_proof2} and boundedness of $\{v^k\}_k$ (cf., Lemma \ref{lem:lem3}), we get $r(u^k) \leq C (\theta_k+\sqrt{\theta_k})$ for some $C>0$, and since $\theta_k\leq 1$, $r(u^k) \leq 2C\sqrt{\theta_k}$. Therefore, there exists a constant $C\geq 0$ such that for every $k \in \mathbb{N}$ with $\theta_k<1$, $\theta_kr(u^k) \geq C r(u^k)^3$. By monotonicity of $\{r(u^k)\}_k$ there exists $C'>0$ such that $C' r(u^k)^2\leq 1$ for all $k \in \mathbb{N}$. In particular, for all $k\in \mathbb{N}$ with $\theta_k =1$, $\theta_k r(u^k) \geq C'r(u^k)^3$. Therefore, for all $k \in \mathbb{N}$,  $\theta_kr(u^k)\geq \min\{C, C'\}r(u^k)^3$. Thus, $\{r(u^k)^3\}_k$ is summable and monotonically non-increasing, hence $r(u^k) = o(k^{-1/3})$, see \cite[Theorem 3.1.1]{knopp1990theory}.
\end{proof}
\begin{remark}
	We believe that the $o(k^{-1/3})$ rate for $r(u^k)$ in Theorem \ref{thm:gap} could be further improved to $o(k^{-1})$, which we leave to a future work.
\end{remark}
The hybrid nature of HPGCG allows us to state a partial convergence result for the iterates.

\begin{thm}\label{thm:convergence_Pu}
	Let $\{u^k\}_k$ be generated by Algorithm \ref{algorithm1} and $P^{1/2}$ be the square root of $P$. Then, $\{P^{1/2}u^k\}_k$ converges weakly to some $p^*=P^{1/2}u^*$, $u^*$ being a minimizer of \eqref{eq:minimization_problem}.
\end{thm}
\begin{proof}
	From part 2) of Lemma \ref{lem1} and Lemma \ref{lem:lem2} it follows that for all $k \in \mathbb{N}$ and every minimizer $u^*$,
	\begin{equation*}
		\tfrac{1}{4}\|u^{k+1}-u^*\|_P^2\leq \tfrac{1}{4}\|u^{k}-u^*\|_P^2+r(u^k)-r(u^{k+1}).
	\end{equation*}
	Thus, for every minimizer $u^*$ of \eqref{eq:minimization_problem}, the sequence $\{\|u^{k}-u^*\|_P^2\}_k$ converges, cf., \cite[Lemma 2, Section 2.2.1]{polyak}. Now, as $\{P^{1/2}u^k\}_k$ is bounded, it admits weak cluster points. Assume that $p^*$ and $p^{**}$ are two such elements with $P^{1/2}u^{k_l}\rightharpoonup p^*$ and $P^{1/2}u^{k_l'}\rightharpoonup p^{**}$. Then, using that $\{u^k\}_k$ is bounded and all its weak cluster points are minimizers of \eqref{eq:minimization_problem} as a consequence of Theorem \ref{thm:1}, it is easy to show that $p^* = P^{1/2}u^*$ and $p^{**} = P^{1/2}u^{**}$ for two minimizers $u^*, \ u^{**}\in H$. Now, since
	\begin{align*}
		\langle Pu^k, u^*-u^{**}\rangle & = \tfrac{1}{2}\|u^{k}-u^{**}\|_P^2 -\tfrac{1}{2}\|u^{k}-u^{*}\|_P^2 \\
		                                & \quad - \tfrac{1}{2}\|u^{**}\|_P^2+\tfrac{1}{2}\|u^{*}\|_P^2,
	\end{align*}
	$\{\langle Pu^k, u^*-u^{**}\rangle\}_k$ converges to some $\lambda \in \mathbb{R}$. Thus, $\{\langle Pu^{k_l}, u^*-u^{**}\rangle\}_l$ and $\{\langle Pu^{k_l'}, u^*-u^{**}\rangle\}_{l}$ converge to $\lambda$ too. Hence, taking the difference and passing to the limit gives $\|u^*-u^{**}\|_P^2=0$, and, thus, $p^* = P^{1/2}u^*=P^{1/2}u^{**} = p^{**}$.
\end{proof}
\section{TV PARAMETER LEARNING}\label{sec:learning}

The automatic tuning of the regularization parameter for regularized inverse problems is an ongoing challenge that recently featured several new data-driven approaches, see, e.g., \cite{ Reyes2}. Here, we propose a new learning model and show that the proposed HPGCG method allows us to solve it efficiently.

Given $p\in \mathbb{N}$ and a $n := p\times p$ grid, to denoise a degraded image $\xi\in \mathbb{R}^{n}$ we consider the classical \emph{ROF} model \cite{RUDIN1992259}
\begin{equation}\label{eq:tv_denoising}
	\min_{u \in \mathbb{R}^n} \ \tfrac{1}{2}\| u - \xi \|^2 + \alpha\text{TV}(u),
\end{equation}
where $\alpha$ is a positive parameter, $\| \cdot \|$ is the $\ell^2$ norm and $\text{TV}$ is the discrete \emph{total variation} functional, namely $\text{TV}(u) := \| \nabla u \|_{1,2}$, where $\nabla$ is the discrete gradient operator defined via standard forward differences and $\| \cdot \|_{1,2}$ is defined by $\|\boldsymbol{v}\|_{1,2} = \sum_{i=1}^n \|\boldsymbol{v}_i\|_2$ for every discrete vector field $\boldsymbol{v}\in \mathbb{R}^{n\times 2}$. For the sake of notation, from now on we often denote
\begin{equation}\label{eq:def_f_and_g}
	f(u):=\tfrac{1}{2}\|u-\xi\|^2, \ \text{and} \  g_\alpha(v):=\alpha\|v\|_{1,2}.
\end{equation}
From standard duality theory, see, e.g., \cite[Section 19.2]{BC} and \cite[Section 6.2.1]{chambollepock}, problem \eqref{eq:tv_denoising} is equivalent to
\begin{equation}\label{eq:tv_denoising_dual}
	\min_{v \in \mathbb{R}^{n\times 2}} \ \tfrac{1}{2} \| \nabla \cdot v + \xi \|^2 \ \text{s.t.:} \ \|v\|_{\infty, 2} \leq \alpha,
\end{equation}
where $\nabla \cdot = -\nabla^*$ is the discrete divergence operator, and $\|\cdot\|_{\infty, 2}$ is the dual norm of $\|\cdot \|_{1,2}$, which is defined for all $v\in \mathbb{R}^{n\times 2}$ by $ \| v\|_{\infty, 2} = \max_{i\in \{1,\dots, n\}} \ \|v_i\|_2$. Optimal solutions $u^\alpha$ and $v^\alpha$ to \eqref{eq:tv_denoising} and \eqref{eq:tv_denoising_dual} respectively are often called \emph{primal-dual pairs} and together can be characterized as solutions of the Fenchel--Rockafellar primal-dual optimality system
\begin{equation}\label{eq:opt_conditions}
	\begin{cases}
		\nabla \cdot v^\alpha = u^\alpha -\xi, \\
		\nabla u^\alpha \in \partial g_\alpha^*(v^\alpha),
	\end{cases}
\end{equation}
as well as the roots of the \emph{primal-dual gap}, which is the non-negative function $\mathcal{G}_\alpha:\mathbb{R}^n\times \mathbb{R}^{n\times 2} \to \mathbb{R}_+$ defined by
\begin{equation}\label{eq:primal_dual_gap}
	\mathcal{G}_\alpha(u,v) := f(u)+g_\alpha(\nabla u)+f^*(\nabla \cdot v)+g_\alpha^*(v),
\end{equation}
where $f^*$ and $g^*_\alpha$ are the \emph{Legendre--Fenchel conjugates} of $f$ and $g_\alpha$ respectively, cf., \cite[Definition 13.1]{BC}. Specifically, in our case, these are $f^*(u) = \tfrac{1}{2}\|u+\xi\|^2-\tfrac{1}{2}\|\xi\|^2$ for all $u \in \mathbb{R}^n$, and, setting $B_{\infty, 2}(\alpha):=\{ v \in \mathbb{R}^{n \times 2} \mid \|v\|_{\infty, 2} \leq \alpha \}$, $g^*_\alpha(v) = \mathbb{I}_{B_{\infty, 2}(\alpha)}(v)$ for all $v \in  \mathbb{R}^{n\times 2}$, where $\mathbb{I}_{ B_{\infty, 2}(\alpha) }$ is the \emph{indicator function} of $B_{\infty, 2}(\alpha)$, i.e., $\mathbb{I}_{B_{\infty, 2}(\alpha)}(v) = 0$ if $v \in B_{\infty, 2}(\alpha)$ and $\mathbb{I}_{B_{\infty, 2}(\alpha)}(v) = +\infty$ else.

\subsection{Learning problem}
Our objective is to learn a function $\alpha: \mathbb{R}^n \to \mathbb{R}_+$ that given a degraded image $\xi\in \mathbb{R}^n$ yields a parameter $\alpha(\xi)$ such that the solution $u^{\alpha(\xi)}$ to \eqref{eq:tv_denoising} is as close as possible to the ground-truth $u^{\dagger}$. The problem can be formulated from a standard machine-learning perspective as follows.

Given a dataset $\mathcal{D}:=\{(u_i^\dagger, \xi_i)\}_{i=1}^N$, where $u_i^\dagger\in \mathbb{R}^n$ is a noise free image (often referred to as ground-truth) and $\xi_i\in \mathbb{R}^n$ is its degraded, or noisy, version, and a suitable space of functions $\mathcal{F}\subset\{\alpha: \mathbb{R}^n \to \mathbb{R}_+\}$, we seek a minimizer of
\begin{equation}\label{eq:learing_pbl_bilevel}
	\min_{\alpha \in \mathcal{F}} \ \frac{1}{N}\sum_{i=1}^N \| u_i^\dagger - u_i^{\alpha(\xi_i)}\|^2,
\end{equation}
where $u_i^{\alpha(\xi_i)}$ is the optimal solution to \eqref{eq:tv_denoising} with data $\xi_i$ as $\xi$ and regularization parameter $\alpha(\xi_i)$ as $\alpha$.

Problem \eqref{eq:learing_pbl_bilevel} has a clear bilevel structure that is not amenable to computation. However, we will see that \eqref{eq:learing_pbl_bilevel} has an elegant connection with the primal-dual gap \eqref{eq:primal_dual_gap}, which can ultimately be used to design a monolevel proxy for \eqref{eq:learing_pbl_bilevel}. In the following result, we show that $\mathcal{G}_\alpha$ in \eqref{eq:primal_dual_gap} can be equivalently expressed as a sum of \emph{Bregman divergences}. Recall that a Bregman divergence relative to a proper, convex, lower semicontinuous function $F$ on a Hilbert space $H$ is defined for all $u,u'\in H$ and $p\in \partial F(u')$ by
\begin{equation*}
	\mathcal{D}^F_p(u,u'):=F(u)-F(u')-\langle p, u-u'\rangle.
\end{equation*}

\begin{thm}\label{thm:gap}
	Let $\alpha >0, \ n\in \mathbb{N}$, let $\nabla: \mathbb{R}^n\to \mathbb{R}^{n\times 2}$ be the discrete gradient operator, and $f,\ g_\alpha$ be defined as in \eqref{eq:def_f_and_g}, then $\mathcal{G}_\alpha$ according to \eqref{eq:primal_dual_gap} admits the following expression
	\begin{align*}
		\mathcal{G}_\alpha(u,v) = \mathcal{D}_{\nabla \cdot v^\alpha}^f(u,u^\alpha)+\mathcal{D}_{u^\alpha}^{f^*}(\nabla \cdot  v,\nabla \cdot v^\alpha) \\
		\hspace{1cm} + \mathcal{D}^{g_\alpha^*}_{\nabla u^\alpha}(v,v^\alpha)+\mathcal{D}_{v^\alpha}^{g_\alpha}(\nabla u,\nabla u^\alpha),
	\end{align*}
	for every $(u,v)\in \mathbb{R}^n\times \mathbb{R}^{n\times 2}$ and every pair of  primal-dual solutions $(u^\alpha, v^\alpha)$.
\end{thm}
\begin{proof}
	The proof follows from straightforward computations, recalling that $\mathcal{G}_\alpha$ vanishes on any primal-dual solution pair.
\end{proof}
In our case, we can see that for all $\alpha >0$,
\begin{equation*}
	\mathcal{D}^f_{\nabla \cdot v^\alpha} (u, u^\alpha) = \tfrac{1}{2} \|u-u^\alpha \|^2,
\end{equation*}
where $u^\alpha$ is the optimal solution to \eqref{eq:tv_denoising}
with parameter $\alpha$ and data $\xi$. Indeed, we have
\begin{align*}
	\mathcal{D}^f_{\nabla \cdot v^\alpha} (u, u^\alpha) & = \tfrac{1}{2}\|u-\xi\|^2-\tfrac{1}{2}\|u^\alpha-\xi\|^2-\langle \nabla \cdot v^\alpha, u-u^\alpha \rangle \\
	                                                    & = \tfrac{1}{2}\|u-(\nabla \cdot v^\alpha +\xi)\|^2 = \tfrac{1}{2} \|u-u^\alpha \|^2,
\end{align*}
where we repeatedly used that $\nabla \cdot v^\alpha +\xi = u^\alpha$, see \eqref{eq:opt_conditions}. Therefore, from Theorem \ref{thm:gap} we have that, for every noise free image $u^\dagger \in \mathbb{R}^n$,
\begin{align*}
	 & \min_{v\in \mathbb{R}^{n\times 2}}\mathcal{G}_\alpha(u^\dagger, v) = \mathcal{G}_\alpha(u^\dagger, v^\alpha)                                                  \\
	 & \quad = \tfrac{1}{2} \|u^\dagger-u^\alpha\|^2+\mathcal{D}^{g_\alpha}_{v^\alpha}(\nabla u^\dagger, \nabla u^\alpha)\geq \tfrac{1}{2} \|u^\dagger-u^\alpha\|^2.
\end{align*}
In particular, for a data point $(u_i^\dagger, \xi_i)\in \mathcal{D}$, the function
\begin{equation}\label{eq:proxy_single}
	\alpha \mapsto \min_{v\in \mathbb{R}^{n\times 2}} \ \mathcal{G}_{\alpha(\xi_i)} (u_i^\dagger, v)
\end{equation}
majorizes the quadratic distance between the reconstructed image $u_i^{\alpha(\xi_i)}$ and the ground-truth $u^\dagger_i$ and can be used to turn \eqref{eq:learing_pbl_bilevel} into the following optimization problem
\begin{equation*}
	\min_{\alpha\in \mathcal{F}, v_1, \dots, v_N\in \mathbb{R}^{n\times 2}} \ \frac{1}{N} \  \sum_{i=1}^N \mathcal{G}_{\alpha(\xi_i)}(u_i^\dagger, v_i),
\end{equation*}
which is equivalent to
\begin{equation}\label{eq:learning_mono1}
	\min_{ (\alpha, \boldsymbol{v})\in \mathcal{C}} \ \frac{1}{2N} \sum_{i=1}^N\|\nabla \cdot v_i+\xi_i\|^2+\frac{1}{N}\sum_{i=1}^N \alpha(\xi_i)\text{TV}(u_i^\dagger),
\end{equation}
where $\mathcal{C}$ is the subset of $\mathcal{F}\times \mathbb{R}^{n\times 2\times N}$ of all $\alpha \in \mathcal{F}$ and $\boldsymbol{v}=(v_1, \dots, v_N) \in \mathbb{R}^{n\times 2\times N}$ such that $\| v_i \|_{\infty, 2} \leq \alpha(\xi_i)$ for all $i \in \{1,\dots, N\}$.

\subsection{Model selection} It remains to fix $\mathcal{F}$. Here, different choices can be made. In this paper, we investigate the performance of quadratic models, i.e., with $\alpha(\xi) = \Bar{\xi}^*A \Bar{\xi}$, where $\Bar{\xi}=[\xi, 1]^*$ and $A$ is a symmetric positive semidefinite matrix of size $(n+1)^2$.

With this choice, problem \eqref{eq:learning_mono1} turns into the following convex problem
\begin{equation}\label{eq:learning_mono2}
	\min_{ (A, \boldsymbol{v})\in \mathcal{C}_q} \ \frac{1}{2N} \sum_{i=1}^N\|\nabla \cdot v_i+\xi_i\|^2+\frac{1}{N}\sum_{i=1}^N \Bar{\xi}_i^*A \Bar{\xi}_i\text{TV}(u_i^\dagger),
\end{equation}
where $\mathcal{C}_q$ is the subset of $\mathbb{R}^{(n+1)\times (n+1)}\times \mathbb{R}^{n\times 2\times N}$ of all positive semidefinite matrices $A \in \mathbb{R}^{(n+1)\times (n+1)}$ and all $\boldsymbol{v}=(v_1, \dots, v_N) \in \mathbb{R}^{n\times 2\times N}$ such that $\| v_i \|_{\infty, 2} \leq \Bar{\xi}_i^*A \Bar{\xi}_i$ for all $i$.

\subsection{Optimization procedure} To solve \eqref{eq:learning_mono2} we employ the HPGCG method with
\begin{align*}
	f(\boldsymbol{v}, A) & := \frac{1}{2N} \sum_{i=1}^N\|\nabla \cdot v_i+\xi_i\|^2+\frac{1}{N}\sum_{i=1}^N \Bar{\xi}_i^*A \Bar{\xi}_i\text{TV}(u_i^\dagger), \\
	g(\boldsymbol{v}, A) & := \mathbb{I}_{\mathcal{C}_q}(v, A).
\end{align*}
Further, we pick $\|(v, A)\|_P^2=\lambda \|A\|_F^2$ for some $\lambda>0$, where $\|\cdot \|_F$ is the Frobenius norm, which from now on we will simply write as $\|\cdot\|$. Note that the projection onto the constraint set $\mathcal{C}_q$ does not admit an explicit expression and can only be computed approximately with possibly time-consuming inner procedures. Such a bottleneck can undermine the convergence performance of standard proximal methods that require the computation of the projection onto $\mathcal{C}_q$. The HPGCG method allows us to circumvent this issue and leads to a low-complexity iterative method, as no projections onto $\mathcal{C}_q$ would be required.

To see this, recall that at each iteration, given $v$ and $A$, we need to solve
\begin{equation*}
	\min_{\bar{v}, \bar{A}} \ \langle \nabla_v f(v,A), \bar{v}\rangle +  \langle \nabla_A f(v,A)-\lambda A, \bar{A}\rangle + \tfrac{\lambda}{2}\|\bar{A}\|^2 + g(\bar{v}, \bar{A}),
\end{equation*}
which in our case reads as
\begin{align}\label{eq:iteration}
	\min_{(\bar{v}, \bar{A})\in \mathcal{C}_q} & \ -\frac{1}{N}\sum_{i=1}^N \langle \nabla (\nabla \cdot v_i +\xi_i), \bar{v}_i\rangle                                                                     \\
	                                           & +\frac{1}{N}\sum_{i=1}^N \text{TV}(u_i^\dagger)\langle \bar{\xi_i}\otimes \bar{\xi_i}-\lambda A, \bar{A}\rangle +\frac{\lambda}{2}\|\bar{A}\|^2.\nonumber
\end{align}
It is easy to observe that if $\widetilde{v}=(\widetilde{v}_1, \dots, \widetilde{v}_N)$ and $\widetilde{A}$ optimize \eqref{eq:iteration} then for all $i\in \{1,\dots, N\}$, $\widetilde{v}_i$ is given for all $j \in \{1, \dots, n\}$ by
\begin{equation}
	(\widetilde{v}_i)_j = \frac{(\nabla (\nabla\cdot v_i+\xi_i))_j}{\|(\nabla (\nabla\cdot v_i+\xi_i))_j\|_2} \bar{\xi_i}^*\widetilde{A}\bar{\xi}_i
\end{equation}
if $\nabla (\nabla\cdot v_i+\xi_i))_j\neq 0$, otherwise it can be chosen as any element such that $\|(\widetilde{v}_i)_j\|_2 \leq  \bar{\xi_i}^*\widetilde{A}\bar{\xi}_i$. Thus, denoting by $c_i = \text{TV}(u_i^\dagger)-\text{TV}(\nabla \cdot v_i+\xi_i)$ for all $i \in \{1, \dots, N\}$, $\widetilde{A}$ actually minimizes
\begin{equation*}
	\min_{\bar{A}\geq 0} \ \bigg\langle \frac{1}{N}\sum_{i=1}^N c_i\Bar{\xi_i}\otimes \Bar{\xi_i}, \bar{A}\bigg\rangle-\langle \lambda A, \bar{A}\rangle +\frac{\lambda}{2}\|\bar{A}\|^2.
\end{equation*}
Therefore, $\widetilde{A}$ is the following projection onto the positive semidefinite cone
\begin{equation*}
	\widetilde{A} = \text{Proj}_+ \left( A-\frac{1}{\lambda N}\sum_{i=1}^N c_i\bar{\xi_i}\otimes \bar{\xi_i}  \right),
\end{equation*}
which can be computed exactly up to numerical tolerances via spectral decomposition, cf., \cite[Section 8.1.1]{boyd}.

Eventually, the proposed HPGCG method when applied to problem \eqref{eq:learning_mono2} turns into the iterative method illustrated in Algorithm \ref{algorithm2}. Note, in particular, that in Algorithm \ref{algorithm2}, by Theorem \ref{thm:convergence_Pu}, we can expect convergence of $\{A^k\}_k$.

\begin{algorithm}[t]
	\SetAlgoLined
	\KwData{$\{(u_i^\dagger, \xi_i)\}_{i=1}^N$ and $\lambda>0$}
	\textbf{Return}: $A^\infty = \lim_{k \to \infty}A^k$\\
	\textbf{Initialize}: $v_1^0, \dots, v_N^0\in \mathbb{R}^{n\times 2}, \ A^0 \in \mathbb{R}^{(n+1)^2}$ \hspace{-0.1cm}with $A^0\geq 0$\\
	\For{$k = 0,1,\ldots$}{
		For all $i\in \{1,\dots, N\}$, set $c_i^k= \text{TV}(u_i^\dagger)-\text{TV}(\nabla \cdot v_i^k+\xi_i)$\\
		Compute $\widetilde{A}^k$ by
		\begin{equation*}
			\widetilde{A}^k = \text{Proj}_+ \left( A^k-\frac{1}{\lambda N}\sum_{i=1}^N c_i^k\Bar{\xi_i}\otimes \Bar{\xi_i}  \right)
		\end{equation*}
		\hspace{-0.2cm}For all $i \in \{1,\dots, N\}$ and  $j \in \{1,\dots, n\}$, if $\nabla (\nabla\cdot v_i^k+\xi_i))_j = 0$ set $(\widetilde{v}^k_i)_j=0$, otherwise,
		$$(\widetilde{v}^k_i)_j = \frac{(\nabla (\nabla\cdot v_i^k+\xi_i))_j}{\|(\nabla (\nabla\cdot v_i^k+\xi_i))_j\|_2}\Bar{\xi_i}^*\widetilde{A}^k\Bar{\xi_i}$$\\
		Compute $\theta_k$ as in \eqref{eq:step-size}\\
		Update $A$ and $v$ as
		\begin{align*}
			A^{k+1}   & = A^k + \theta_k(\widetilde{A}^k-A^k)        \\
			v_i^{k+1} & = v_i^k + \theta_k (\widetilde{v}^k_i-v_i^k)
		\end{align*}
	}
	\caption{HPGCG for solving
		problem (\ref{eq:learning_mono2}).}\label{algorithm2}
\end{algorithm}

\subsection{Numerical experiments}
In this section we present our numerical experiments\footnote{All computations were carried out in
	Python on a PC with 62 GB RAM and an Intel Core i7-9700 CPU@3.00GHz. Data and code can be found at \href{https://github.com/TraDE-OPT/TV-parameter-learning}{https://github.com/TraDE-OPT/TV-parameter-learning}}. We train our model considering a dataset of $N\in \mathbb{N}$ patches of size $n = p\times p$ with $p=16$. Working with small patches instead of full pictures allows us to consider significantly more data points than degrees of freedom, which are of order $O(n^2)$. This would allow us to avoid overfitting phenomena. Specifically, we consider a dataset of  $101440$ patches extracted from $1109$ cartoon images, and to each patch we apply a Gaussian noise of variance $0.05$. We set $\lambda = 50$ and run Algorithm \ref{algorithm2} choosing $D_f(u, v) = L/2\|u-v\|^2$ with $L = 8/N$.

\begin{remark}
	Note that in Algorithm \ref{algorithm2} at iteration $k$ the term $\nabla \cdot v_i^k + \xi_i$ is an approximation of the solution of \eqref{eq:tv_denoising} with data $\xi_i$, i.e., an approximation of the TV-denoised $i^{th}$ patch, which can be used to monitor the reconstruction quality online, see Figure \ref{fig:patches}.
\end{remark}

At every iteration, in order to update the step-size we should compute the residual $D(A^k, v^k)$, which can also be employed for a stopping rule. Specifically, we stop the iteration as soon as $D(A^k, v^k) < 10^{-4}$ (reached in 19294 iterations). The residual as a function of the iteration number is shown in Figure \ref{fig:residual}.

\textit{Experiment 1.} We employ the trained model to denoise $6$ new test images, which we split into $16\times 16$ patches. For every single patch $\xi_i$ we compute the TV-parameter by $\bar{\xi}_i^* A^{\infty} \bar{\xi_i}$ where $A^{\infty}$ is given by the trained model. In Figure \ref{fig:parameters}, we can see that the proposed model adaptively yields higher values for flatter regions (e.g., the backgrounds) and lower values for more complex parts of the images, as expected.

\textit{Experiment 2.} In this experiment, we assess the performance of the proposed model in a systematic comparison with more naive choices, e.g., fixed constant parameters. We consider a test set of $N_t = 200$ patches extracted arbitrarily from $281$ cartoon images. As performance metrics, we first consider the Mean Squared Error (MSE) relative to the parameters, namely, for each test patch $\xi_i$ we compute via HPGCG the best parameter $\alpha^*_i$ according to \eqref{eq:learning_mono1}, with $N=1$ and $\mathcal{F}$ composed of only non-negative constants, for $100000$ iterations (up to a residual of $\sim 10^{-5}$). Then, we measure the Mean Squared Error, i.e.,
\begin{equation}\label{eq:MSE_par_def}
	\text{MSE}_\alpha := \frac{1}{N_t} \sum_{i=1}^{N_t} (\alpha_i^*-\alpha(\xi_i))^2.
\end{equation}
We also compute the MSE with respect to eight constant choices spaced evenly from $10^{-4}$ to $10^{-1}$, i.e., \eqref{eq:MSE_par_def} replacing $\alpha(\xi_i)$ with these constant values. Further, we consider a constant model trained on $1000$ patches extracted from the same training set via HPGCG, with $\lambda = 50$ and stopped as soon as the residual drops below $10^{-5}$ (reached in $15479$ iterations). The constant model yielded a value of $\alpha = 2.713 \ 10^{-2}$.

Eventually, for each parameter choice (computed with our trained model or given by one of the constants above), we also measure the Mean Squared Error relative to the reconstructed images, namely
\begin{equation}\label{eq:MSE_img_def}
	\text{MSE}_u := \frac{1}{N_t} \sum_{i=1}^{N_t} \| u_i^\dagger-u_i^{\alpha(\xi_i)} \|^2,
\end{equation}
or \eqref{eq:MSE_img_def} replacing $\alpha(\xi_i)$ with the above constant values. The results are contained in Table \ref{tab:comparisons}.

\begin{table}[t]
	\centering
	\begin{tabular}{l@{\ }l@{\ }l@{\ }l@{\ }l@{\ }l}
		\toprule
		\multirow{2}{*}{Models} & \multirow{2}{*}{Quadratic}                             & \multicolumn{4}{c}{Constant $\alpha = \eta \ 10^{-4}$}                                                             \\
		\cmidrule(lr){3-6}
		{}                      & {}                                                     & $\eta = 1$                                             & $\eta = 2.68$     & $\eta = 7.20$     & $\eta = 19.3$     \\
		\midrule
		$\text{MSE}_\alpha$     & $\mathbf{3.39} \ \mathbf{10^{-4}}$                     & $18.56 \ 10^{-4}$                                      & $18.42 \ 10^{-4}$ & $18.08 \ 10^{-4}$ & $17.17 \ 10^{-4}$ \\
		$\text{MSE}_u$          & $\mathbf{0.1529}$                                      & $0.6374$                                               & $0.6312$          & $0.6148$          & $0.5729$          \\
		\bottomrule
		\toprule
		\multirow{2}{*}{Models} & \multicolumn{5}{c}{Constant $\alpha = \eta \ 10^{-3}$}                                                                                                                      \\
		\cmidrule(lr){2-6}
		{}                      & $\eta = 5.18$                                          & $\eta = 13.9$                                          & $\eta = 27.13$    & $\eta = 37.3$     & $\eta = 100$      \\
		\midrule
		$\text{MSE}_\alpha$     & $14.88 \ 10^{-4}$                                      & $9.79 \ 10^{-4}$                                       & $4.95 \ 10^{-4}$  & $3.62 \ 10^{-4}$  & $41.08 \ 10^{-4}$ \\
		$\text{MSE}_u$          & $0.4737$                                               & $0.2917$                                               & $0.1833$          & $0.1777$          & $0.4764$          \\
		\bottomrule
	\end{tabular}
	\caption{}
	\label{tab:comparisons}
\end{table}

\textit{Results.} From Figure \ref{fig:residual} we can see that HPGCG, before entering into a sub-linear regime, is able to quickly reach high precision within about a few hundreds of iterations. From Table \ref{tab:comparisons} we can also see that the proposed model yields very accurate parameter choices and performs better than constant models in terms of MSE both relative to the parameter choice and the image reconstruction.

\begin{figure}
	\centering \hspace{-1cm}
	\includegraphics[width = 0.81\linewidth]{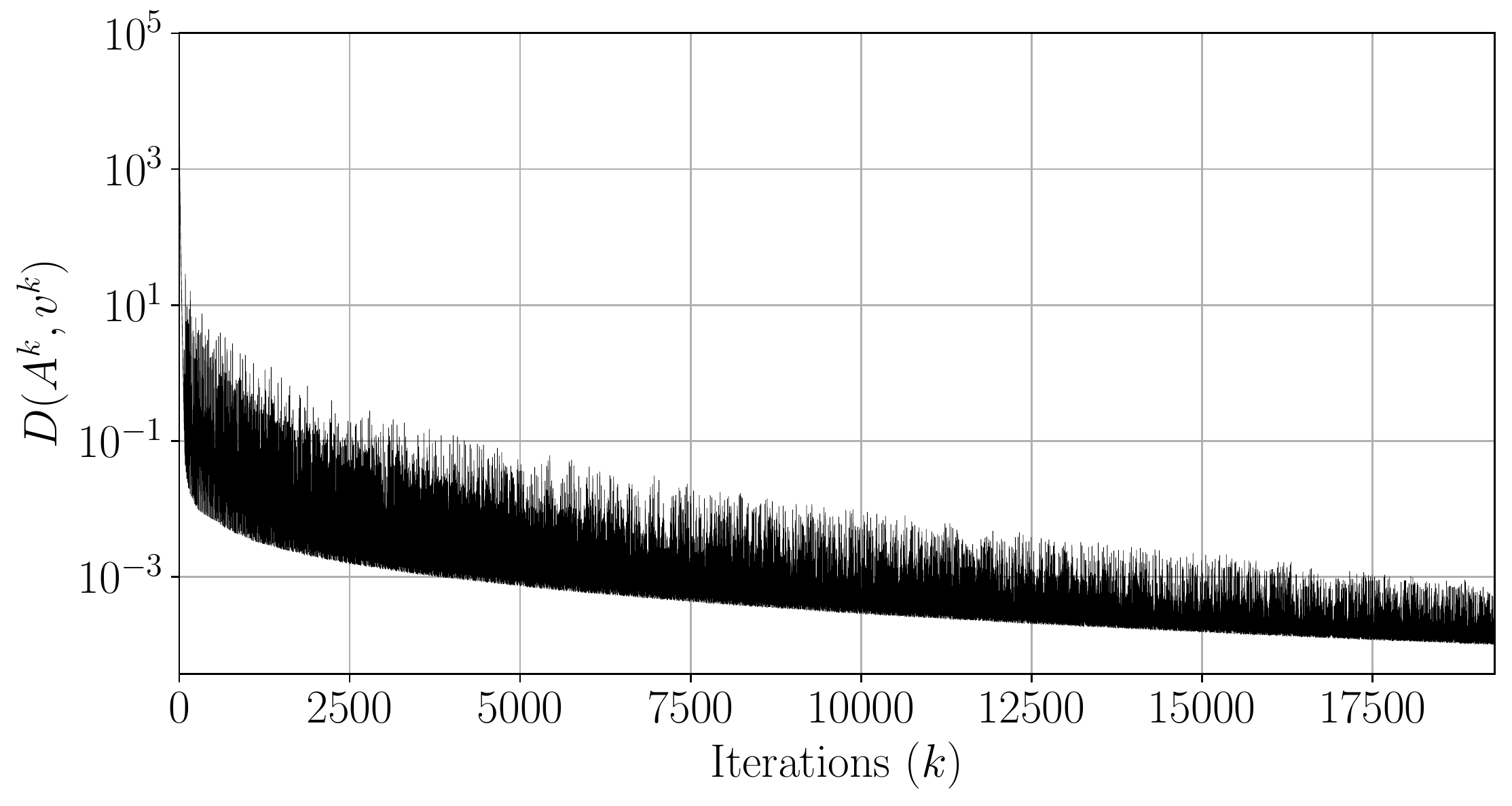}
	\caption{Residual as a function of the iteration number.}
	\label{fig:residual}
\end{figure}

\begin{figure}
	\centering
	\hspace{0.27cm}
	\includegraphics[width = 0.92\linewidth]{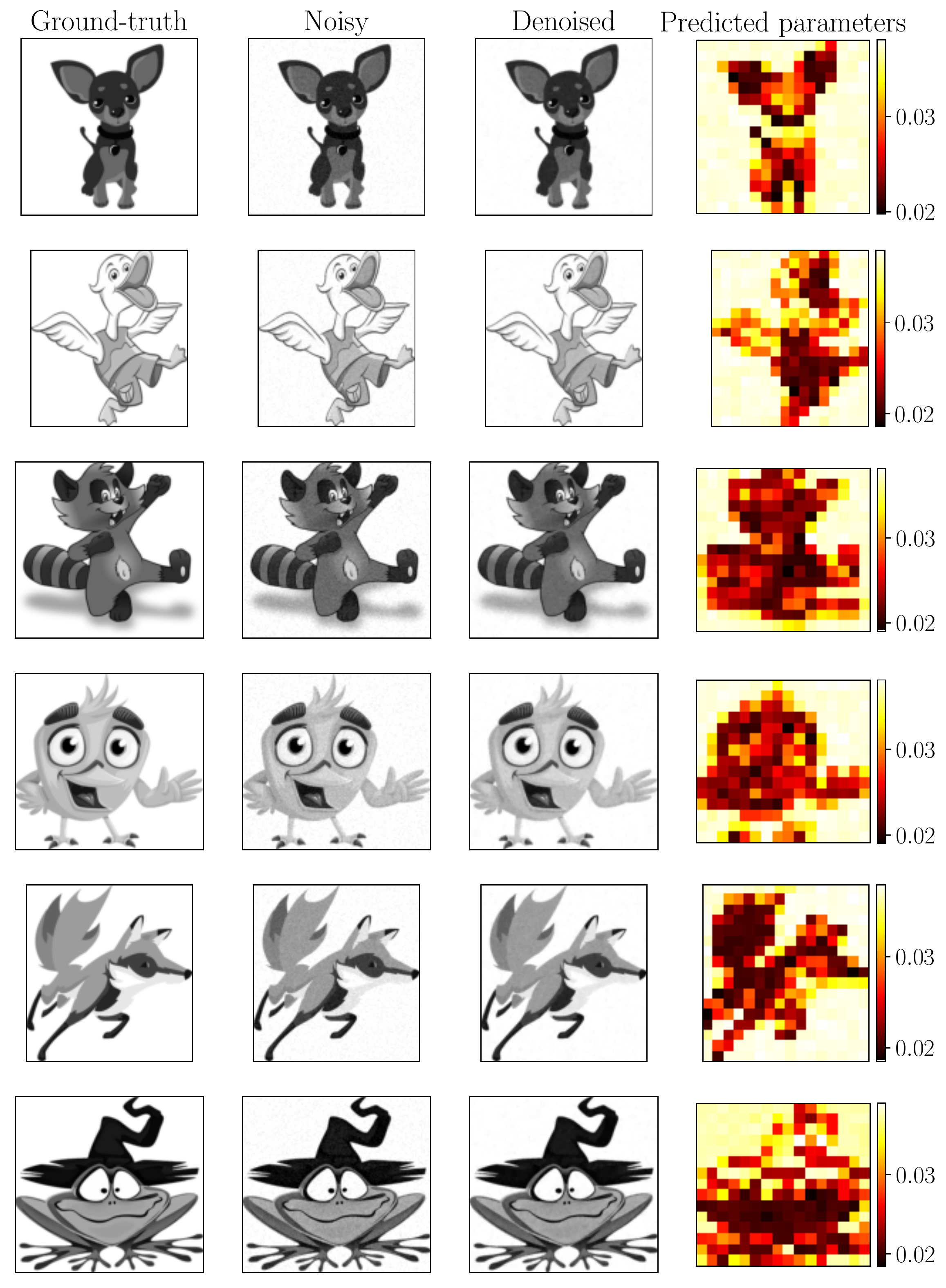}
	\caption{Ground-truth, noisy and denoised images using predicted parameters for every patch.}
	\label{fig:parameters}
\end{figure}

\begin{figure}
	\centering
	\begin{minipage}{0.42\linewidth}
		\includegraphics[width=1\linewidth]{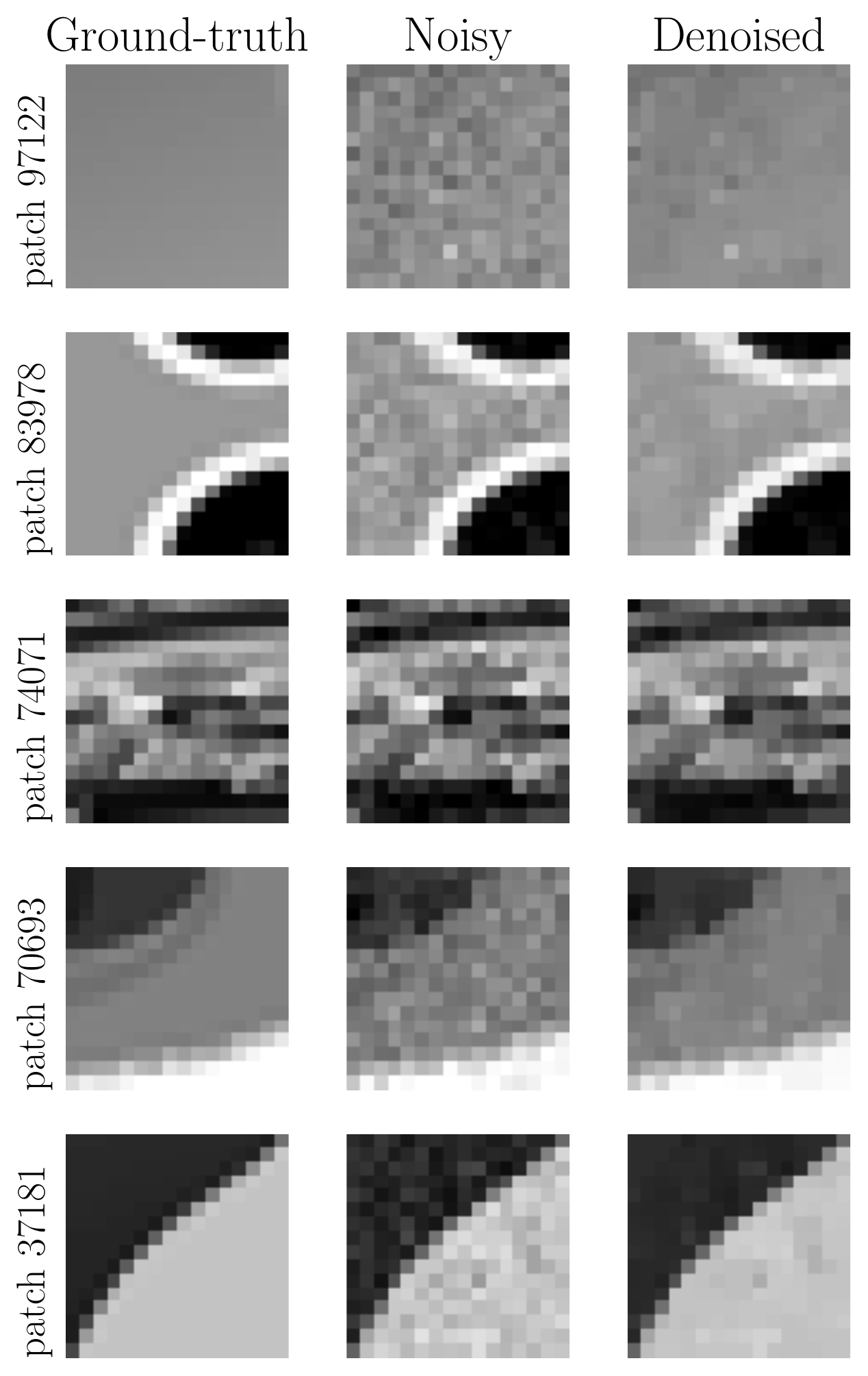}
	\end{minipage}
	\begin{minipage}{0.42\linewidth}
		\includegraphics[width=1\linewidth]{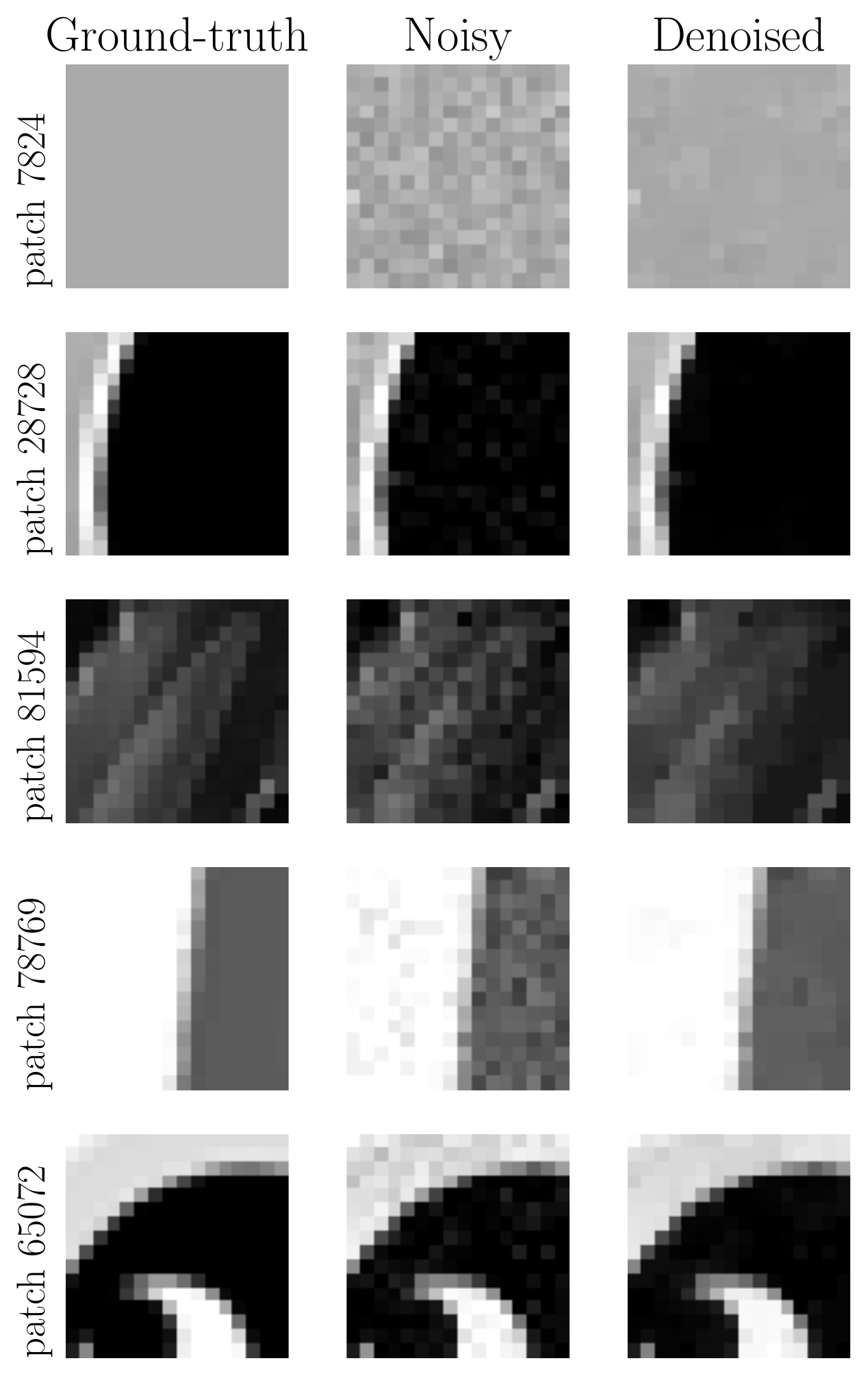}
	\end{minipage}
	\caption{Some ground-truth, noisy and denoised patches observed online.}
	\label{fig:patches}
\end{figure}

\addtolength{\textheight}{-12cm}   %

\section*{ACKNOWLEDGMENTS}

K.B. and E.C. have received funding from the European Union’s Framework Programme for Research and Innovation Horizon 2020 (2014--2020) under the Marie Skłodowska-Curie Grant Agreement No. 861137 \includegraphics[width=0.05\linewidth]{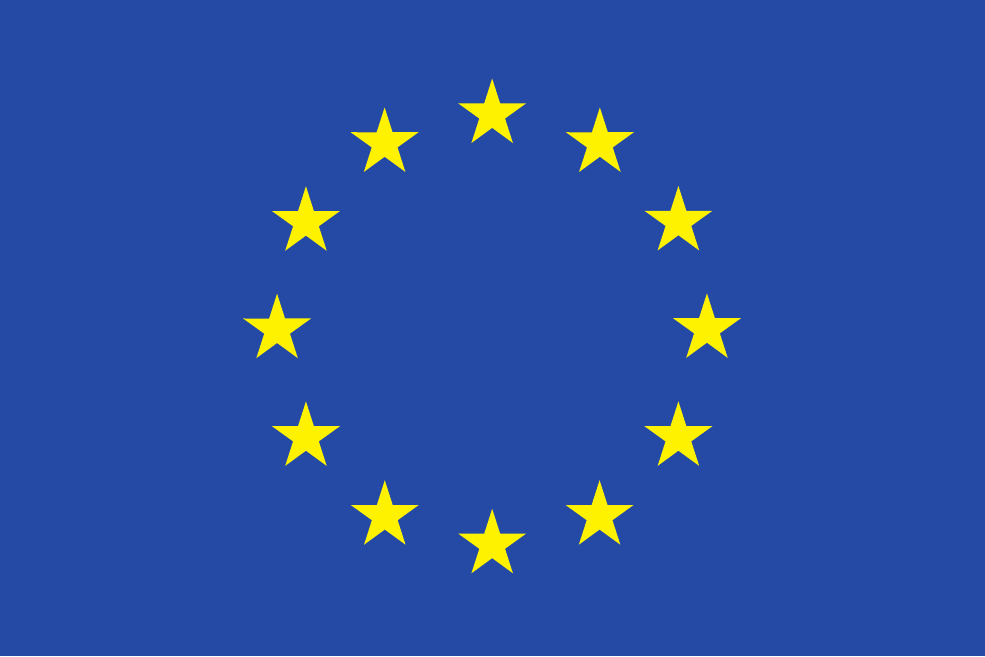}\ .
The Institute of Mathematics and Scientific Computing, to which K.B. and E.C. are affiliated, is a member of NAWI Graz (\href{https://www.nawigraz.at/en}{https://www.nawigraz.at/}).

\bibliographystyle{IEEEtran}

\end{document}